\documentclass{svproc}

\usepackage{graphicx}
\usepackage{amsmath, amssymb}
\usepackage{booktabs}
\DeclareMathOperator{\rk}{rk}
\usepackage{url}

\makeatletter
\let\ps@plain\ps@empty
\makeatother

\begin{document}
\mainmatter  

\title{Structural Redundancy in Subspace Network Coding via Atomic Decompositions}
\titlerunning{Subspace Network Coding via Redundant Atomic Decompositions}

\author{%
David Ramirez\inst{1} \and
Elvis Cabrera\inst{2} \and
Jyrko Correa-Morris\inst{1}
}

\authorrunning{David Ramirez et al.}

\institute{%
Miami Dade College, 627 SW 27th Ave, Miami, FL 33135, USA \\
\email{david.ramirez033@mymdc.net, jcorrea7@mdc.edu}
\and
Massachusetts Institute of Technology, Cambridge, MA 02139, USA \\
\email{elviscal@mit.edu}
}

\maketitle

\begin{abstract}
Random linear network coding (RLNC) provides a powerful framework for non-coherent communication, where reliable transmission requires correcting errors and erasures induced by network mixing and motivates the use of subspace codes. In this work, we introduce an atomic perspective on subspace coding by formalizing the notion of minimal atomic decompositions in the lattice $\mathcal{L}(V)$ of subspaces of a finite-dimensional vector space over a finite field. We study the function $\mathfrak{N}$ that assigns to each subspace the number of its minimal atomic decompositions and establish its key structural properties.

Leveraging $\mathfrak{N}$, we define a new distance metric on $\mathcal{L}(V)$ that refines classical subspace comparisons by capturing atomic-level overlap. We then introduce the \emph{Atomic Operator Channel}, a transmission model for RLNC in which codewords are conveyed through atomic decompositions and corruption is modeled via atomic insertions and erasures. Within this framework, we prove a minimum-distance decoding guarantee for the induced metric. In the constant-dimension setting, we show that the classical unique-decodability condition under the subspace distance remains sufficient for unique decoding under the atomic metric.

\end{abstract}

\section{Introduction}
Reliable communication of linear information over unreliable channels is a central problem in modern communication theory. In many contemporary systems\allowbreak---such as network coding,
distributed storage, and wireless multicast\allowbreak---information is naturally modeled not as a sequence of symbols, but as a linear subspace of a vector space over a finite field. This perspective has led to a rich body of work on subspace channels and subspace codes, where distortion is typically quantified through dimension loss or rank-based metrics.

Within this framework, classical models describe channel impairments through subspace erasures and insertions, and decoding performance is analyzed using distances that depend only on the dimensions of transmitted and received subspaces and their intersection. These approaches have proven highly effective and have led to powerful coding constructions and capacity results. Nevertheless, they treat subspaces as indivisible objects and deliberately abstract away their internal structure.

From a geometric viewpoint, however, a subspace possesses nontrivial internal combinatorial structure. In particular, a subspace may admit many distinct minimal decompositions into one-dimensional subspaces (atoms). Existing models regard all such decompositions as equivalent, even though channel impairments may affect different atomic directions in fundamentally different ways. As a consequence, two subspaces that are indistinguishable under classical dimension-based distances may differ substantially in their internal atomic structure, a distinction that current models are unable to capture.

Motivated by this observation, we introduce a new channel abstraction, which we refer to as the \emph{Atomic Operator Channel}. In this model, the channel acts directly on the atomic structure of the transmitted subspace: individual one-dimensional subspaces may be deleted, and new one-dimensional subspaces may be inserted. The receiver observes only the resulting subspace after these atomic operations. Importantly, atomic deletions do not necessarily reduce the dimension of the transmitted subspace, and atomic insertions may introduce structured distortions that are invisible to purely dimension-based models.

To analyze this channel, we propose a decomposition-based invariant. For a subspace $S$, let $\mathfrak{N}(S)$ denote the number of minimal atomic decompositions of $S$. We show that this function is supermodular on the lattice of subspaces. This structural property allows us to define a distance
\[
d_{\mathfrak{N}}(S,T) = \mathfrak{N}(S) + \mathfrak{N}(T) - 2\mathfrak{N}(S \cap T),
\]
and to prove that $d_\mathfrak{N}$ is a metric on the lattice of subspaces of a vector space over a finite field $\mathbb{F}$. Unlike classical subspace distances, this metric quantifies distortion in terms of the loss of decomposition richness rather than solely through dimension. As a result, it enables a refined notion of redundancy, where robustness is encoded in the multiplicity of minimal atomic realizations rather than in dimension alone.

Within this framework, coding and decoding can be formulated using standard minimum-distance principles with respect to the proposed metric. Redundant atomic structure plays a role analogous to redundancy in classical coding theory, but at the level of one-dimensional subspaces rather than individual symbols. This provides a new perspective on robustness against structured degradations that complements existing subspace-based models.

The main contributions of this paper are as follows:
\begin{enumerate}
    \item We introduce the Atomic Operator Channel, a subspace channel model that captures atomic-level erasures and insertions.
    \item We define an invariant based on minimal atomic decompositions and prove its supermodularity on the subspace lattice.
    \item We construct a decomposition-based metric on subspaces and establish its fundamental properties.
    \item We show how this metric enables a refined notion of redundancy and minimum-distance decoding beyond dimension-based approaches.
\end{enumerate}

The remainder of the paper is organized as follows. 
Section~\ref{sec:preliminaries} reviews the relevant background on subspace network coding. 
Section~\ref{sec:minimal_atomic_decompositions} investigates the main structural properties of the function $\mathfrak{N}$, which assigns to each subspace $S$ the total number of its minimal atomic decompositions. 
Section~\ref{sec:atomic_channel} formalizes the Atomic Operator Channel. 
Section~\ref{sec:metric} introduces the $\mathfrak{N}$-based metric and establishes its fundamental properties. 
Section~\ref{sec:codeundermetric} analyzes the main properties of codes under the $\mathfrak{N}$-based metric. 
Section~\ref{sec:minimum-distance-decoder} discusses the corresponding minimum-distance decoder. 
Section~\ref{sec:singleton-bound} proposes a Singleton-type bound with respect to the $\mathfrak{N}$-metric. 
Section~\ref{sec:comparison} compares the $\mathfrak{N}$-metric with dimension-based metrics, and Section~\ref{sec:conclusion} concludes the paper.

\section{Related Work} \label{sec:preliminaries}
Random linear network coding (RLNC) has attracted sustained attention as a robust
strategy for transmitting information in \emph{non-coherent} communication
networks, where the internal structure and dynamics of the network are unknown.
In RLNC, intermediate nodes do not merely forward packets; instead, they transmit
random linear combinations of the packets they receive over a finite field
$\mathbb{F}_q$. While this mixing mechanism enables efficient operation under
packet loss and changing routes, it also fundamentally alters the nature of error
control: perturbations introduced at any point in the network propagate through
subsequent linear combinations. As a result, a central objective in RLNC is the
development of reliable \emph{error and erasure correction} principles that remain
valid in this algebraic, non-coherent setting.

A key observation is that, from the receiver’s viewpoint, the information
contained in the collected linear combinations is naturally captured by the
subspace they span rather than by individual vectors. This motivates
\emph{subspace coding} as the appropriate framework for reliability in RLNC
\cite{koetter2008coding}. The foundational work of R.~K\"otter and
F.~R.~Kschischang \cite{koetter2008coding} formalized this viewpoint through the
operator channel model, in which a transmitted subspace may lose dimensions
(erasures) and gain an additional error component (insertions). Decoding then
amounts to recovering the transmitted subspace from the received one, with
geometric structure entering explicitly via distances on projective spaces and
Grassmannians.

Since the introduction of the operator channel, the study of subspace codes has
developed into a broad area combining algebraic, geometric, and combinatorial
techniques. A particularly prominent direction concerns \emph{constant dimension
codes} (CDCs), where all codewords have the same dimension. CDCs are attractive
both for their clean geometric structure and for their compatibility with many
RLNC protocols. This line of research includes deep connections to rank-metric
coding and corresponding decoding viewpoints \cite{silva2008rank,bartz2022rankmetric},
explicit families with strong optimality properties such as spread-based
constructions \cite{manganiello2008spread,manganiello2011spread}, and algorithmic
questions including efficient message encoding and retrieval for spread and
orbit-type codes \cite{trautmann2014message,horlemann2018message}. Further advances
include inserting-type and multilevel constructions \cite{heinlein2023improved,li2025inserting},
recent developments for related families such as flag codes \cite{han2025flag},
and comprehensive surveys of constructions and bounds for subspace codes
\cite{kurz2025bounds}.

In much of this literature, distortion is measured by the classical
\emph{subspace distance} introduced in \cite{koetter2008coding},
$$d(U,V)=\dim(U)+\dim(V)-2\dim(U\cap V),$$
or by closely related dimension-based metrics such as the injection distance.
These distances underlie minimum-distance decoding guarantees and many standard
bounds and constructions for subspace codes, and they interact naturally with
models of insertions and erasures, including approaches that combine subspace
codes with rank-metric techniques \cite{horlemann2023insertion}. However, such
metrics depend only on dimensional parameters and do not capture finer internal
structure within subspaces.

This abstraction is both deliberate and effective: by collapsing internal
structure into dimension and intersection dimension, classical models provide a
\allowbreak tractable and robust framework for analysis. At the same time, this coarsening
implies that two subspaces of equal dimension and equal intersection dimension
with a received subspace are regarded as equally distant, even if their internal
generating structures differ substantially. In particular, dimension-based
metrics are insensitive to \emph{directional losses} that do not reduce
dimension—for example, the loss of specific one-dimensional components that are
replaced by alternative generating directions.

From an operational perspective, such losses may significantly affect
robustness or recoverability, especially when redundancy is encoded internally
rather than through increased dimension. Classical subspace coding typically
treats redundancy as an external resource, achieved by enlarging the ambient
dimension or increasing separation under dimension-based metrics. In contrast,
redundancy arising from multiple minimal representations of the same subspace—an
intrinsic combinatorial property—is invisible to these models. As a consequence,
two subspaces of equal dimension may exhibit markedly different resilience to
structured perturbations while remaining indistinguishable under existing
distances.

These observations do not indicate deficiencies in existing subspace models, but
rather delineate the scope of their abstraction. Dimension-based distances are
well suited to scenarios in which degradation is assumed to be generic or
unstructured. However, in settings where impairments act sparsely or
directionally, or where redundancy is encoded through multiple minimal
representations, a more refined description becomes relevant. This motivates the
atomic perspective adopted in this work, which preserves the subspace-based
framework while exposing internal structural features that are invisible to
purely dimension-based approaches.

We also relate the proposed framework to the established geometry of constant-dimension coding. 
In the constant-dimension regime, we show that the $\mathfrak{N}$-induced distance is compatible with the classical correction paradigm: whenever a received subspace admits a unique nearest codeword under subspace-distance decoding, it is also uniquely correctable under $d_{\mathfrak{N}}$-based decoding. 

At the same time, codes with varying codeword dimensions have been studied extensively, both from the perspective of constructions and from the standpoint of upper bounds and fundamental limitations \cite{kurz2025bounds,honold2019johnson}. 
This setting is particularly relevant in non-coherent communication, where mixed-dimension effects arise naturally and many known bounds are still formulated in terms of dimension-based metrics. 

In this mixed-dimension context, we demonstrate that the atomic viewpoint can distinguish competing candidates that are indistinguishable under dimension-only information, yielding explicit scenarios in which atomic-structure decoding succeeds while classical metrics lack sufficient resolution.

\section{Minimal Atomic Decompositions in $\mathcal{L}(V)$}
\label{sec:minimal_atomic_decompositions}

Let $\mathcal L(V)$ denote the lattice of subspaces of an $m$-dimensional vector
space $V$ over the finite field $\mathbb{F}_q$, ordered by inclusion, with join operator $\vee$
given by subspace sum and meet operator $\wedge$ given by intersection.
That is, for subspaces $S_1,S_2 \subseteq V$,
\(S_1 \vee S_2 = S_1 + S_2\) and 
\(S_1 \wedge S_2 = S_1 \cap S_2
\).
The atoms of $\mathcal L(V)$, denoted by $\mathcal A$, are the one-dimensional
subspaces of $V$.

\subsection{Atomic Decompositions}

\begin{definition}
Let $S \subseteq V$ be a subspace.
A subset
\(
\mathfrak a_S = \{A_1,A_2,\ldots,A_k\} \subseteq \mathcal A
\)
is called an \emph{atomic decomposition} of $S$ if $\mathfrak a_S$ spans $S$, i.e.,
\[
\bigvee_{i=1}^k A_i=A_1 + A_2 + \cdots + A_k = S.
\]
The decomposition $\mathfrak a_S$ is said to be \emph{minimal} if no proper subset
of $\mathfrak a_S$ is an atomic decomposition of $S$.
\end{definition}

Minimal atomic decompositions correspond to irredundant representations of a
subspace in terms of one-dimensional components and play a central role in the
structural analysis developed in this paper.

\subsection{Characterization of Minimal Atomic Decompositions} 

\begin{lemma}
A subset $\mathfrak a_S=\{A_1,A_2,\dots,A_k\}\subset \mathcal A$ is a minimal atomic
decomposition of the subspace $S$ if and only if every choice of vectors
$\{v_1,v_2,\ldots,v_k\}$ with $v_i\in A_i$ and $v_i\neq 0$ forms a basis of $S$.
\end{lemma}

\begin{proof}
Suppose $\mathfrak a_S=\{A_1,\dots,A_k\}$ is a minimal atomic decomposition of $S$.
Then necessarily
\(
S = A_1 \oplus \cdots \oplus A_k.
\)
Let $u\in S$ admit two decompositions
$u=v_1+\cdots+v_k=\hat v_1+\cdots+\hat v_k$ with $v_i,\hat v_i\in A_i$.
Setting $w_i=v_i-\hat v_i\in A_i$, we have $\sum_i w_i=0$.
If some $w_{i_0}\neq 0$, then $A_{i_0}$ lies in the span of the remaining atoms,
contradicting minimality.
Hence all $w_i=0$, so the decomposition is unique.
It follows that any choice of nonzero vectors $v_i\in A_i$ is linearly independent
and spans $S$, and thus forms a basis.

Conversely, if every such choice $\{v_1,\ldots,v_k\}$ forms a basis of $S$, then
$\mathfrak a_S$ spans $S$, and no proper subset can do so.
Thus $\mathfrak a_S$ is minimal.
\end{proof}

\subsection{Counting Minimal Atomic Decompositions}

Let $\mathcal B$ denote the set of all unordered bases of $S$.
Define an equivalence relation $\sim$ on $\mathcal B$ by declaring
\(
B \sim B'
\quad\text{if and only if}\quad
\{\langle v\rangle : v\in B\} = \{\langle v'\rangle : v'\in B'\}.
\)
That is, two bases are equivalent if they determine the same collection of
one-dimensional subspaces.

\begin{definition}
For a subspace $S\subseteq V$, let $\mathfrak N(S)$ denote the number of minimal
atomic decompositions of $S$.
\end{definition}

\begin{lemma}\label{lem:N-count}
Let $S$ be an $n$-dimensional subspace of $V$.
Then
\[
\mathfrak N(S) = |\mathcal B/\!\sim| = \frac{|\mathcal B|}{(q-1)^n}.
\]
\end{lemma}

\begin{proof}
Each equivalence class in $\mathcal B/\!\sim$ corresponds to a unique minimal
atomic decomposition of $S$.
Conversely, each minimal atomic decomposition determines $(q-1)^n$ unordered
bases obtained by choosing nonzero generators in each one-dimensional atom.
Thus $|\mathcal B/\!\sim| = |\mathcal B|/(q-1)^n$.
\end{proof}

\begin{theorem}
Let $S$ be an $n$-dimensional subspace of $V$ over $\mathbb F_q$.
Then
\[
\mathfrak N(S)
= \frac{\prod_{i=0}^{n-1}(q^n - q^i)}{n!\,(q-1)^n}.
\]
\end{theorem}

\begin{proof}
The number of ordered bases of $S$ is $|GL_n(\mathbb F_q)|=\prod_{i=0}^{n-1}(q^n-q^i)$.
Dividing by $n!$ accounts for unordered bases, and dividing by $(q-1)^n$ accounts
for scalar rescalings of basis vectors within each atom.
The result follows from Lemma~\ref{lem:N-count}.
\end{proof}
\begin{corollary}\label{Isotonic} $\mathfrak{N}$ is an isotone function; that is, for all subspaces $S,T\subseteq V$, $S\subseteq T$ implies $\mathfrak{N}(S)\le \mathfrak{N}(T)$.
\end{corollary}

\begin{proof}
    Every minimal atomic decomposition of $S$ can be extended to a minimal atomic decomposition of $T$. This defines an injective function from the set of minimal atomic decompositions of $S$ to the set of minimal atomic decompositions of $T$. Hence, $\mathfrak{N}(S)\le \mathfrak{N}(T)$. 
\end{proof}

\section{The Atomic Operator Channel}\label{sec:atomic_channel}
We consider communication between a transmitter and a receiver occurring over
sequential transmission rounds through a network with intermediate nodes.
At each round, the transmitter injects into the network a collection of packets,
all of the same length, which traverse the network and may be altered by the action
of intermediate nodes before reaching the receiver.
Rather than modeling the detailed packet-level operations performed by the network,
we focus on the algebraic object induced by a transmission round.

Specifically, packets are modeled as vectors in $\mathbb{F}_q^n$ for some finite
field $\mathbb{F}_q$, and the information conveyed in a transmission round is
identified with the subspace spanned by these packets.
Classical subspace-based models treat the transmitted object solely as a subspace
of an ambient space and characterize channel impairments through dimension loss
or insertion of additional dimensions.
In contrast, our goal is to capture distortions that act on the internal structure
of a subspace.

\subsection{Atomic Representation of Transmissions}

Let $V$ be a finite-dimensional vector space over $\mathbb{F}_q$, and let
$S \subseteq V$ denote the subspace transmitted in a given round.
A central feature of our framework is that $S$ is not represented by a basis,
but rather through a redundant atomic decomposition of $S$ ---
a collection 
\(
\mathfrak{a}_S = \{A_1, \dots, A_k\},
\)
where each $A_i$ is a one-dimensional subspace of $S$ and
\(
S = A_1 + \dots + A_k.
\)
The cardinality $k>\dim S$ since $\mathfrak{a}_S$ is required to  be redundant.

We model a transmission through the network as an operation acting directly on
the atomic representation of $S$.
Specifically, if $\mathfrak{a}_S$ is the atomic decomposition input into the
network, the receiver observes an atomic decomposition
\[
\mathfrak{a}_U = \mathfrak{a}_S' \cup \{B_1, \dots, B_t\},
\]
where $\mathfrak{a}_S' \subseteq \mathfrak{a}_S$ represents atoms that survive
the transmission (atomic erasures), and $\{B_1, \dots, B_t\}$ are one-dimensional
subspaces not contained in $S$, representing atomic insertions.
The received subspace is then given by
\[
U = \bigvee_{A \in \mathfrak{a}_U} A.
\]

This model shifts the focus from subspaces viewed as indivisible objects to their
internal atomic composition.
In particular, atomic erasures need not reduce the dimension of the transmitted
subspace, and atomic insertions may introduce perturbations that are
not captured by purely dimension-based models.

\subsection{Redundancy and Structural Robustness}

An essential component of our framework is the deliberate use of redundancy.
Rather than transmitting a basis of $S$, the transmitter employs atomic
decompositions that include additional atoms beyond a minimal generating set.
The purpose of this redundancy is to increase robustness against erasures and
other perturbations that may occur as the message propagates through the
network.

At each node, the atomic representation of the message may be altered through
the removal of atoms or the insertion of new ones.
By transmitting a decomposition with controlled redundancy, the probability
that all information necessary to reconstruct $S$ is lost due to atom removal
is significantly reduced, since a generating core may survive even when a
portion of the atoms is erased.
In practice, the number of atoms is chosen as a fixed multiple of $\dim S$,
providing a tunable trade-off between transmission overhead and resilience to
structural degradation.

\subsection{The Atomic Operator Channel}

\begin{definition}[Atomic Operator Channel]
Let $V$ be a finite-dimensional vector space over $\mathbb{F}_q$, and let
$S \subseteq V$ be a transmitted subspace.
Let
\[
\mathfrak{a}_S = \{A_1, A_2, \dots, A_k\}
\]
be a redundant atomic decomposition of $S$, where each $A_i$ is a one-dimensional
subspace of $S$ and
\(
S= \bigvee_{i=1}^k A_i.
\)
The cardinality $k>\dim S$, allowing for redundant
atomic representations.

An \emph{atomic operator channel} produces an output subspace
$U \subseteq V$ of the form
\[
U
\;=\;
\bigvee_{A_i \in \mathfrak{a}_S'} A_i
\;\vee\;
\bigvee_{j=1}^{t} B_j,
\]
where:
\begin{enumerate}
    \item $\mathfrak{a}_S' \subseteq \mathfrak{a}_S$ denotes the collection of
    surviving atoms. The elements in $\mathfrak{a}_S\setminus\mathfrak{a}_S'$ correspond to the atomic erasures. The number of atomic erasures is $k - |\mathfrak{a}_S'|$.
    \item $\{B_1,\dots,B_t\}$ is a collection of one-dimensional subspaces of $V$
    such that $B_j \cap S = \{0\}$ for all $j$, modeling atomic insertions.
    \item The error space is $E = \bigvee_{j=1}^{t} B_j$.
\end{enumerate}

The receiver observes only the resulting subspace $U$.
This model captures channel distortions at the level of atomic structure,
where erasures correspond to the removal of one-dimensional components and
insertions correspond to the introduction of new one-dimensional directions
outside the transmitted subspace.
\end{definition}

\section{The $\mathfrak{N}$-Induced Distance}
\label{sec:metric}

We now establish the structural properties of the function $\mathfrak N$
that enable the construction of a distance on the lattice of subspaces.
Throughout this section, $V$ denotes a finite-dimensional vector space
over $\mathbb F_q$.

\subsection{Supermodularity of $\mathfrak N$}

\begin{theorem}\label{thm:supermodular}
For all subspaces $S,T \subseteq V$, the function $\mathfrak N$ satisfies
\begin{description}
     \item[(i)] $\mathfrak N(S) \cdot \mathfrak N(T)
\;\le\;
\mathfrak N(S \cap T) \cdot \mathfrak N(S + T)$, \quad (log-supermodularity)
\item[(ii)] $\mathfrak N(S) + \mathfrak N(T)
\;\le\;
\mathfrak N(S \cap T) + \mathfrak N(S + T).$ \quad (supermodularity)
\end{description}

\end{theorem}
\begin{proof}
    \item[(i)] We aim to prove that for all \( S, W \in L(V) \),
		\[
		\mathfrak{N}(S) \cdot \mathfrak{N}(W) \leq \mathfrak{N}(S + W) \cdot \mathfrak{N}(S \cap W),
		\]
		or equivalently,
		\[
		\frac{\mathfrak{N}(S + W)/\mathfrak{N}(W)}{\mathfrak{N}(S)/\mathfrak{N}(S \cap W)} \geq 1.
		\]
		
		Let us analyze both sides using the known formula:
		\[
		\mathfrak{N}(X) = \frac{\prod_{i=0}^{n_X - 1}(q^{n_X} - q^i)}{n_X! \cdot (q - 1)^{n_X}},
		\]
		where \( n_X = \dim(X) \). By factoring terms, we have:
		\begin{align*}
		\prod_{i=0}^{n_{S \vee W} - 1}(q^{n_{S \vee W}} - q^i)
		&= q^{n_W(n_{S \vee W} - n_W)} \cdot \prod_{i=0}^{n_W - 1}(q^{n_W} - q^i) \cdot \prod_{i=0}^{t - 1}(q^{n_{S \vee W}} - q^i), \\
		\prod_{i=0}^{n_S - 1}(q^{n_S} - q^i)
		&= q^{n_{S \wedge W}(n_S - n_{S \wedge W})} \cdot \prod_{i=0}^{n_{S \wedge W} - 1}(q^{n_{S \wedge W}} - q^i) \cdot \prod_{i=0}^{t - 1}(q^{n_S} - q^i),
		\end{align*}
		where \( t := n_{S \vee W} - n_W = n_S - n_{S \wedge W} \).
		
		Thus, the quotient becomes:
		\[
		\frac{\mathfrak{N}(S \vee W)/\mathfrak{N}(W)}{\mathfrak{N}(S)/\mathfrak{N}(S \wedge W)} = 
		\frac{
			q^{n_W t} \cdot n_W! \cdot \prod_{i=0}^{t - 1}(q^{n_{S \vee W}} - q^i)
		}{
			(q - 1)^t \cdot n_{S \vee W}!
		}
		\cdot
		\frac{
			(q - 1)^t \cdot n_S!
		}{
			q^{n_{S \wedge W} t} \cdot n_{S \wedge W}! \cdot \prod_{i=0}^{t - 1}(q^{n_S} - q^i)
		}.
		\]
		
		Now observe:
		\begin{itemize}
			\item The exponents satisfy \( n_W t - n_{S \wedge W} t = t(n_W - n_{S \wedge W}) \).
			\item Since \( n_{S \vee W} \geq n_S \), it follows that
			\[
			\prod_{i=0}^{t - 1}(q^{n_{S \vee W}} - q^i) \geq \prod_{i=0}^{t - 1}(q^{n_S} - q^i).
			\]
		\end{itemize}
		
		Therefore, the whole expression is bounded below by
		\[
		q^{t(n_W - n_{S \wedge W})} \cdot \frac{n_S! \cdot n_W!}{n_{S \vee W}! \cdot n_{S \wedge W}!}.
		\]
		
		Now we claim that for all \( t \leq n_S \),
		\[
		q^{t(n_W - n_S) + t^2} \geq \frac{(n_W + t)! \cdot (n_S - t)!}{n_W! \cdot n_S!}.
		\]
		We prove this by induction on \( t \).
		
		\emph{Base case:} For \( t = 0 \),
		\[
		q^0 = 1 = \frac{n_W! \cdot n_S!}{n_W! \cdot n_S!}.
		\]
		
		\emph{Inductive step:} Assume the inequality holds for \( t = k \):
		\[
		q^{k(n_W - n_S) + k^2} \geq \frac{(n_W + k)! \cdot (n_S - k)!}{n_W! \cdot n_S!}.
		\]
		
		We must prove it for \( t = k + 1 \), that is:
		\[
		q^{(k+1)(n_W - n_S) + (k+1)^2} \geq \frac{(n_W + k + 1)! \cdot (n_S - k - 1)!}{n_W! \cdot n_S!}.
		\]
		
		Dividing both sides of the \( t = k+1 \) inequality by the inductive hypothesis, it suffices to show:
		\[
		q^{(n_W + k + 1) - (n_S - k)} \geq \frac{n_W + k + 1}{n_S - k}.
		\]
		
		Since \( q \geq 2 \), the inequality
		\[
		q^{a-b} \geq \frac{a}{b}
		\]
		holds for all \( a,b \in \mathbb{N} \) such that $a>b$, and thus the desired inequality follows. 

        \item[(ii)] It is a well-known fact that if a function is isotone Corollary \ref{Isotonic} and log-supermodular, then it is also supermodular.
\end{proof}

Supermodularity formalizes the idea that the combinatorial complexity of
atomic decompositions increases when subspaces are combined and decreases
when structure is shared.

\subsection{Definition of the $\mathfrak N$-Induced Distance}

\begin{definition}
Let $S,T \subseteq V$ be subspaces of $V$.
The \emph{$\mathfrak N$-induced distance} between $S$ and $T$ is defined by
\[
d_{\mathfrak N}(S,T)
=
\mathfrak N(S) + \mathfrak N(T) - 2\,\mathfrak N(S \cap T).
\]
\end{definition}

This quantity measures the loss of atomic decomposition structure when
restricting from $S$ and $T$ to their intersection subspace.

\subsection{Metric Properties}

\begin{theorem}
The function $d_{\mathfrak N} : \mathcal L(V) \times \mathcal L(V) \to \mathbb R$
is a metric on the lattice of subspaces.
\end{theorem}

\begin{proof}
\emph{Non-negativity.}
Since $S \cap T \subseteq S$ and $S \cap T \subseteq T$, and $\mathfrak N$ is
isotone, we have
\[
\mathfrak N(S \cap T) \le \min\{\mathfrak N(S),\mathfrak N(T)\},
\]
which implies $d_{\mathfrak N}(S,T) \ge 0$.

\emph{Symmetry.}
The definition is symmetric in $S$ and $T$.

\emph{Identity of indiscernibles.}
If $S=T$, then $S \cap T = S$ and hence $d_{\mathfrak N}(S,T)=0$.
Conversely, if $d_{\mathfrak N}(S,T)=0$, then $\mathfrak N(S)-\mathfrak N(S \cap T) = \mathfrak N(S \cap T)-\mathfrak N(T)$.
By isotonicity, this can occur only if $S=S\cap T=T$, and hence $S=T$.

\emph{Triangle inequality.}
Let $S,T,$ and $U$ be subspaces of $V$.
Recall that
\begin{align*}
d_{\mathfrak{N}}(S,U) &= \mathfrak{N}(S) + \mathfrak{N}(U) - 2\mathfrak{N}(S\cap U),\\
d_{\mathfrak{N}}(S,T) &= \mathfrak{N}(S) + \mathfrak{N}(T) - 2\mathfrak{N}(S\cap T),\\
d_{\mathfrak{N}}(T,U) &= \mathfrak{N}(T) + \mathfrak{N}(U) - 2\mathfrak{N}(T\cap U).
\end{align*}

Thus, to prove that
\(
d_{\mathfrak{N}}(S,U) \le d_{\mathfrak{N}}(S,T) + d_{\mathfrak{N}}(T,U),
\)
it suffices to verify that
\begin{equation}\label{eq:triangle-key}
\mathfrak{N}(S\cap U) + \mathfrak{N}(T) \;\ge\; \mathfrak{N}(S\cap T) + \mathfrak{N}(T\cap U).
\end{equation}

By the supermodularity of $\mathfrak{N}$, we have
\begin{equation}\label{eq:supermod}
\mathfrak{N}(S\cap T) + \mathfrak{N}(T\cap U)
\;\le\;
\mathfrak{N}\bigl((S\cap T)\cap (T\cap U)\bigr)
+
\mathfrak{N}\bigl((S\cap T) + (T\cap U)\bigr).
\end{equation}

Notice that
\(
(S\cap T)\cap (T\cap U) = S\cap T\cap U \subseteq S\cap U,
\)
and hence, by isotonicity of $\mathfrak{N}$,
\begin{equation}\label{eq:ineq1}
\mathfrak{N}\bigl((S\cap T)\cap (T\cap U)\bigr) \le \mathfrak{N}(S\cap U).
\end{equation}

Moreover,
\(
(S\cap T) + (T\cap U) \subseteq T,
\)
and therefore,
\begin{equation}\label{eq:ineq2}
\mathfrak{N}\bigl((S\cap T) + (T\cap U)\bigr) \le \mathfrak{N}(T).
\end{equation}

Combining \eqref{eq:supermod}, \eqref{eq:ineq1}, and \eqref{eq:ineq2}, we obtain
\[
\mathfrak{N}(S\cap T) + \mathfrak{N}(T\cap U)
\;\le\;
\mathfrak{N}(S\cap U) + \mathfrak{N}(T),
\]
which proves \eqref{eq:triangle-key}. Consequently, $d_{\mathfrak{N}}$ satisfies the triangle inequality.

\end{proof}

\subsection{Computational complexity of evaluating $d_{\mathfrak N}$}

We measure complexity in \emph{field operations} over $\mathbb{F}_q$.
Let $S,T \subseteq \mathbb{F}_q^{\,n}$ be given by generator matrices
$A_S \in \mathbb{F}_q^{m_S\times n}$ and $A_T \in \mathbb{F}_q^{m_T\times n}$.

We use the standard fact that for $A\in\mathbb{F}_q^{m\times n}$,
$\rk(A)$ can be computed by Gaussian elimination using
$O\!\big(mn\min\{m,n\}\big)$ field operations. Moreover,
\[
\dim(S\cap T)=\dim(S)+\dim(T)-\dim(S+T),
\qquad
\dim(S+T)=\rk\!\begin{pmatrix}A_S\\A_T\end{pmatrix}.
\]

\begin{proposition}
Given generator matrices $A_S,A_T$ for the subspaces $S,T\subseteq \mathbb{F}_q^{\,n}$,
the distance $d_{\mathfrak N}(S,T)$ can be evaluated using $O(n^3)$
field operations over $\mathbb{F}_q$.
\end{proposition}

\begin{proof}
Compute
\[
d_S=\rk(A_S),\qquad d_T=\rk(A_T),\qquad
d_+=\rk\!\begin{pmatrix}A_S\\A_T\end{pmatrix}.
\]
In the worst case $m_S\le n$, $m_T\le n$, and $m_S+m_T\le 2n$, so each rank
computation costs $O(n^3)$ field operations and the total remains $O(n^3)$.
Then $\dim(S\cap T)=d_S+d_T-d_+$, and
\[
d_{\mathfrak N}(S,T)=\mathfrak N(d_S)+\mathfrak N(d_T)-2\,\mathfrak N(d_S+d_T-d_+).
\]
The remaining arithmetic consists of evaluating $\mathfrak{N}(d_S)$, $\mathfrak{N}(d_T)$, and $\mathfrak{N}(d_I)$. By the closed-form characterization of $\mathfrak{N}$, each evaluation $\mathfrak{N}(d)$ depends only on the integer $d \in \{0,1,\dots,n\}$ and can be computed in $O(d)\le O(n)$ arithmetic steps (e.g., as a product of $d$ factors), or in $O(1)$ time per query after precomputing the table $\{\mathfrak{N}(0),\dots,\mathfrak{N}(n)\}$. Hence, this arithmetic contributes at most $O(n)$ additional work and does not affect the dominant $O(n^3)$ term coming from the rank computations.
\end{proof}

The values $\mathfrak N(d)$ may be very large. If one measures bit-complexity,
then integer arithmetic/comparisons depend on the output size; however, in the
field-operation model the dominant cost remains the rank computations.

\section{Codes Under the $\mathfrak{N}$-Induced Distance} \label{sec:codeundermetric}

Let $V$ be an $N$-dimensional ambient vector space over $\mathbb{F}_q$.
A \emph{code} for the Atomic Operator Channel is a nonempty collection
\[
\mathcal C \subseteq \mathcal L(V)
\]
of subspaces of $V$.
Equivalently, a code may be viewed as a collection of subspaces together with
their associated redundant atomic representations used at
transmission time.

The \emph{size} of the code is $|\mathcal C|$.
Distances between codewords are measured using the $\mathfrak N$-induced
distance $d_{\mathfrak N}$ defined in the previous section.

\subsection{Minimum Distance}

The \emph{minimum distance} of a code $\mathcal C$ is defined as
\[
D(\mathcal C)
\;=\;
\min_{\substack{X,Y\in\mathcal C \\ X\neq Y}}
d_{\mathfrak N}(X,Y).
\]
As in classical coding theory, the minimum distance governs the error-correction
capability of the code under minimum-distance decoding.
In particular, larger values of $D(\mathcal C)$ correspond to greater
separation between codewords in terms of atomic decomposition structure.

\subsection{Atomic Error Measures}

Let $S \in \mathcal C$ denote the transmitted subspace and let $U \subseteq V$
denote the received subspace produced by the Atomic Operator Channel.
Let
\[
S' = S \cap U
\]
denote the surviving portion of $S$ contained in $U$.

\begin{definition}
The atomic distortion between $S$ and $U$ is quantified as follows:
\begin{enumerate}
    \item The \emph{atomic erasure cost} is
    \[
    \epsilon(S,S') := d_{\mathfrak N}(S,S'),
    \]
    measuring the loss of atomic decomposition structure from $S$ to its
    surviving component $S'$.
    
    \item The \emph{atomic insertion cost} is
    \[
    \iota(U,S') := d_{\mathfrak N}(U,S'),
    \]
    measuring the contribution of atomic insertions outside $S$.
    
    \item The \emph{total atomic error} is
    \[
    \delta(U,S) := d_{\mathfrak N}(U,S),
    \]
    capturing the overall distortion between the transmitted and received
    subspaces.
\end{enumerate}
\end{definition} 

These quantities provide a refined decomposition of channel distortion into
erasures and insertions at the atomic level.
Unlike classical dimension-based error measures, atomic erasure and insertion
costs need not correspond directly to changes in subspace dimension and may
reflect structural degradation that remains invisible under traditional models.

\section{A Minimum Distance Decoder}\label{sec:minimum-distance-decoder}

The $\mathfrak N$-induced distance naturally gives rise to a minimum-distance
decoding strategy.
Let $\mathcal C \subseteq \mathcal L(V)$ be a code for the Atomic Operator Channel.
A \emph{minimum-distance decoder} with respect to $d_{\mathfrak N}$ is defined as
a decoder that, upon receiving a subspace $U \subseteq V$, outputs a codeword
$\widehat S \in \mathcal C$ satisfying
\[
d_{\mathfrak N}(U,\widehat S)
\;=\;
\min_{T\in\mathcal C} d_{\mathfrak N}(U,T).
\]

We now show that, as in classical coding theory, such a decoder successfully
recovers the transmitted codeword whenever the total atomic distortion is less
than half the minimum distance of the code.

\begin{theorem}[Minimum-Distance Decoding]\label{thm:error-erasure-correction}
Let $V$ be an $N$-dimensional vector space over the finite field $\mathbb F_q$, and
let $\mathcal C \subseteq \mathcal L(V)$ be a code with minimum distance
$D(\mathcal C)$ under the $\mathfrak N$-induced distance.
Suppose that $S\in\mathcal C$ is transmitted over an atomic operator channel, and
that the received subspace is
\[
U
=
\bigvee_{A_i \in \mathfrak a_S'} A_i
\;\vee\;
\bigvee_{j=1}^{t} B_j,
\qquad U \subseteq V,
\]
where $\mathfrak a_S'$ denotes the surviving atoms of $S$ and the inserted atoms
$\{B_j\}$ satisfy $B_j \cap S = \{0\}$.
Let
\[
S' = S \cap U
\]
denote the surviving portion of the transmitted subspace.

Then:
\begin{enumerate}
\item The total atomic distortion satisfies
\begin{align*}
d_{\mathfrak N}(U,S)
&=
d_{\mathfrak N}(S,S') + d_{\mathfrak N}(U,S')\\
&=\epsilon(S,S') +  \iota(U,S')
\end{align*}

\item If
\[
2\,d_{\mathfrak N}(U,S) < D(\mathcal C),
\]
then $S$ is the unique codeword minimizing the distance to $U$, and the
minimum-distance decoder recovers $S$.
\end{enumerate}
\end{theorem}

\begin{proof}
Since $S' \subseteq S$ and $S' \subseteq U$, and all inserted atoms are disjoint
from $S$, we have
\[
S \cap S' = S' \quad\text{and}\quad U \cap S' = S'.
\]
Therefore,
\begin{align*}
d_{\mathfrak N}(S,S') + d_{\mathfrak N}(S',U)
&= \mathfrak N(S) + \mathfrak N(S') - 2\mathfrak N(S') \\
&\quad + \mathfrak N(S') + \mathfrak N(U) - 2\mathfrak N(S') \\
&= \mathfrak N(S) + \mathfrak N(U) - 2\mathfrak N(S') \\
&= d_{\mathfrak N}(S,U),
\end{align*}
which proves the first claim.

For the second claim, let $T\in\mathcal C$ with $T\neq S$.
By the triangle inequality,
\[
d_{\mathfrak N}(S,T)
\le
d_{\mathfrak N}(S,U) + d_{\mathfrak N}(U,T).
\]
Since $d_{\mathfrak N}(S,T) \ge D(\mathcal C)$, it follows that
\[
d_{\mathfrak N}(U,T)
\ge
D(\mathcal C) - d_{\mathfrak N}(U,S).
\]
If $2\,d_{\mathfrak N}(U,S) < D(\mathcal C)$, then
\[
d_{\mathfrak N}(U,T) > d_{\mathfrak N}(U,S)
\]
for all $T\neq S$, and hence $S$ is the unique minimizer of the distance to $U$.
\end{proof}

\section{A Singleton-Type Bound Under the $\mathfrak N$-Distance}\label{sec:singleton-bound}

Recall that for any $n$-dimensional subspace $S$ over $\mathbb F_q$, the number of
minimal atomic decompositions depends only on $n$.
We therefore write
\[
f(n) := \mathfrak N(S)\qquad ( \dim S = n).
\]
In particular, for $X,Y\in\mathcal L(V)$,
\[
d_{\mathfrak N}(X,Y)
=
f(\dim X)+f(\dim Y)-2f(\dim(X\cap Y)).
\]
When $\mathcal C\subseteq \mathcal G_q(N,k)$ is a constant-dimension code, this simplifies to
\[
d_{\mathfrak N}(X,Y)=2\bigl(f(k)-f(\dim(X\cap Y))\bigr).
\]

\begin{theorem}[Singleton-type bound induced by $d_{\mathfrak N}$]\label{thm:singleton}
Let $V$ be an $N$-dimensional vector space over $\mathbb F_q$, and let
$\mathcal C\subseteq \mathcal G_q(N,k)$ be a constant-dimension code with minimum
$\mathfrak N$-distance
\[
D(\mathcal C)=\min_{\substack{X,Y\in\mathcal C\\X\neq Y}} d_{\mathfrak N}(X,Y).
\]
Define
\[
s^\star
:=
\max\Bigl\{s\in\{0,1,\dots,k\}:\; 2\bigl(f(k)-f(s)\bigr)\ge D(\mathcal C)\Bigr\},
\qquad
\delta_{\mathrm{eff}}:=k-s^\star.
\]
Then every pair of distinct codewords satisfies $\dim(X\cap Y)\le s^\star$, and
\[
|\mathcal C|
\;\le\;
\left[\!\!\begin{array}{c}
N-\delta_{\mathrm{eff}}+1\\[2pt]
k
\end{array}\!\!\right]_q.
\]
Equivalently,
\[
|\mathcal C|
\;\le\;
\left[\!\!\begin{array}{c}
N-\delta_{\mathrm{eff}}+1\\[2pt]
N-k
\end{array}\!\!\right]_q.
\]
\end{theorem}

\begin{proof}
Let $X\neq Y$ be codewords. Since $d_{\mathfrak N}(X,Y)\ge D(\mathcal C)$ and
$d_{\mathfrak N}(X,Y)=2(f(k)-f(\dim(X\cap Y)))$, we obtain
\[
2\bigl(f(k)-f(\dim(X\cap Y))\bigr)\ge D(\mathcal C),
\]
hence $\dim(X\cap Y)\le s^\star$ by definition of $s^\star$. Set
$\delta_{\mathrm{eff}}=k-s^\star$.

Now puncture each codeword by intersecting with a fixed $(N-\delta_{\mathrm{eff}}+1)$-dimensional
subspace $H\subseteq V$. Standard arguments for constant-dimension subspace codes
show that the map $X\mapsto X\cap H$ is injective under the constraint
$\dim(X\cap Y)\le k-\delta_{\mathrm{eff}}$.
Therefore, $|\mathcal C|$ is at most the number of $k$-subspaces of $H$, namely
$\bigl[\!\!\begin{smallmatrix}N-\delta_{\mathrm{eff}}+1\\ k\end{smallmatrix}\!\!\bigr]_q$.
The dual form follows from $\bigl[\!\!\begin{smallmatrix}n\\k\end{smallmatrix}\!\!\bigr]_q
=\bigl[\!\!\begin{smallmatrix}n\\n-k\end{smallmatrix}\!\!\bigr]_q$.
\end{proof}

\section{Comparison with Dimension-Based Decoding}\label{sec:comparison}

We compare minimum-distance decoding under the classical subspace distance with
minimum-distance decoding under the $\mathfrak N$-induced distance. We show that,
for constant-dimension codes, successful decoding under the classical metric
implies successful decoding under the atomic metric.
\begin{definition}
Let $\mathcal C \subseteq \mathcal G_q(N,k)$ be a constant-dimension code, and
let $S,T \subseteq \mathbb F_q^n$ be subspaces.
\begin{enumerate}
\item $M_1$ denotes a minimum-distance decoder for $\mathcal C$ with respect to
the classical subspace distance
\[
d(S,T)=\dim(S)+\dim(T)-2\dim(S\cap T).
\]

\item $M_2$ denotes a minimum-distance decoder for $\mathcal C$ with respect to
the $\mathfrak N$-induced distance
\[
d_{\mathfrak N}(S,T)=\mathfrak N(S)+\mathfrak N(T)-2\,\mathfrak N(S\cap T).
\]
\end{enumerate}
\end{definition}

\begin{proposition}\label{prop:M1-implies-M2}
Let $\mathcal C \subseteq \mathcal G_q(N,k)$ be a constant-dimension code.
If $M_1$ correctly reconstructs the transmitted codeword from a received
subspace $U$, then $M_2$ also reconstructs it correctly from $U$.
\end{proposition}

\begin{proof}
Let $S\in\mathcal C$ be the transmitted codeword and let $U$ be the received
subspace. Since $M_1$ reconstructs $S$, for every $S_i\in\mathcal C$ with
$S_i\neq S$ we have
\begin{equation}\label{eq:D1-ineq}
d(U,S) < d(U,S_i).
\end{equation}
Because $\mathcal C$ has constant dimension $k$, inequality
\eqref{eq:D1-ineq} becomes
\[
\dim(U)+k-2\dim(U\cap S)
<
\dim(U)+k-2\dim(U\cap S_i).
\]
Canceling $\dim(U)+k$ and dividing by $-2$ yields
\begin{equation}\label{eq:dim-int-D1}
\dim(U\cap S) > \dim(U\cap S_i)
\qquad\text{for all } S_i\neq S.
\end{equation}

Suppose, for contradiction, that $M_2$ outputs some $S_j\in\mathcal C$ with
$S_j\neq S$. Since $M_2$ is a minimum-distance decoder under $d_{\mathfrak N}$,
we have
\begin{equation}\label{eq:D2-wrong}
d_{\mathfrak N}(U,S_j) < d_{\mathfrak N}(U,S).
\end{equation}
Expanding $d_{\mathfrak N}$ and canceling the common term $\mathfrak N(U)$,
inequality \eqref{eq:D2-wrong} becomes
\[
\mathfrak N(S_j)-2\mathfrak N(U\cap S_j)
<
\mathfrak N(S)-2\mathfrak N(U\cap S).
\]
Since $\dim(S_j)=\dim(S)=k$, we have $\mathfrak N(S_j)=\mathfrak N(S)$, and
hence
\[
\mathfrak N(U\cap S_j) > \mathfrak N(U\cap S).
\]
Because $\mathfrak N(\cdot)$ is strictly increasing with respect to dimension,
this implies
\begin{equation}\label{eq:dim-int-D2}
\dim(U\cap S_j) > \dim(U\cap S),
\end{equation}
which contradicts \eqref{eq:dim-int-D1} applied to $S_i=S_j$. Therefore, $M_2$
cannot output any codeword other than $S$, and must correctly reconstruct $S$.
\end{proof}

Now we present an explicit example in which the classical minimum-distance decoder $M_1$ fails to identify the transmitted subspace, while the atomic decoder $M_2$ successfully reconstructs the original message.

\begin{example}

Let $S_1, S_2,$ and $U$ be subspaces of $\mathbb{F}_{257}^4 = \langle e_1,e_2,e_3,e_4 \rangle$.
Consider the two codewords
\[
S_1 = \langle e_1, e_2 \rangle, \qquad
S_2 = \langle e_1, e_2, e_3 \rangle,
\]
so that $S_1 \subseteq S_2$.

Assume $S_1$ is transmitted through the Atomic Operator Channel as follows: one erasure removes the atom $\langle e_2 \rangle$ from $S_1$, and one insertion adds the atom $\langle e_3 \rangle$ (with $\langle e_3 \rangle \cap S_1 = \{0\}$). Thus the received subspace is
\[
U = \langle e_1, e_3 \rangle.
\]

\emph{Decoder based on the dimension distance $(d)$.}
Recall that
\[
d(S,T) = \dim(S) + \dim(T) - 2\dim(S \cap T).
\]

Then
\[
d(U,S_1) = 2 + 2 - 2\dim(\langle e_1 \rangle) = 2 + 2 - 2 = 2.
\]

Since $U \subseteq S_2$, we have $\dim(U \cap S_2) = \dim(U) = 2$, hence
\[
d(U,S_2) = 2 + 3 - 2 \cdot 2 = 1.
\]

Therefore, the minimum-distance decoder with respect to $d$ outputs $S_2$ (since $1 < 2$), even though the transmitted codeword was $S_1$.

\emph{Decoder based on minimal atomic decompositions $(d_{\mathfrak N})$.}
Define
\[
d_{\mathfrak N}(S,T)\;=\;\mathfrak N(\dim S)\;+\;\mathfrak N(\dim T)\;-\;2\,\mathfrak N\!\bigl(\dim(S\cap T)\bigr),
\]
where $\mathfrak N(t)$ denotes the number of minimal atomic decompositions of any $t$-dimensional subspace.

Let $A_1,A_2,A_3$ be subspaces of ${F}_{257}^4$ such that
\[
\dim(A_1)=1,\qquad \dim(A_2)=2,\qquad \dim(A_3)=3.
\]
Using the precomputed values
\[
\mathfrak N(A_1)=1,\qquad
\mathfrak N(A_2)=33153,\qquad
\mathfrak N(A_3)= 48,397,976,536,193
\]
we obtain
\[
d_{\mathfrak N}(U,S_1)
=\mathfrak N(A_2)+\mathfrak N(A_2)-2\mathfrak N(A_1)
=33153+33153-2
=66{,}304.
\]
Moreover, since $U\subseteq S_2$, we have $U\cap S_2=U$, so
\begin{align*}
d_{\mathfrak N}(U,S_2)
&=\mathfrak N(A_2)+\mathfrak N(A_3)-2\mathfrak N(A_2)
=48,397,976,536,193-33{,}153\\
&=48,397,976,503,040.
\end{align*}
Thus,
\[
d_{\mathfrak N}(U,S_1)=66{,}304<48,397,976,503,040=d_{\mathfrak N}(U,S_2),
\]
so the minimum-distance decoder with respect to $d_{\mathfrak N}$ correctly outputs $S_1$.
\end{example}

\section{Conclusion}\label{sec:conclusion}

In this work we introduced an atomic perspective on subspace coding for random
linear network coding, motivated by the observation that classical
dimension-based models do not fully capture how structured perturbations affect
subspaces in non-coherent networks. By formalizing the notion of minimal atomic
decompositions in the lattice of subspaces, we identified an intrinsic
combinatorial invariant—the number of minimal atomic decompositions—that reflects
internal redundancy and structural richness beyond dimension alone.

Leveraging this invariant, we defined a new metric on the lattice of subspaces
and introduced the Atomic Operator Channel, a transmission model in which
corruption acts through atomic-level insertions and erasures. We showed that the
resulting distance admits a natural minimum-distance decoding strategy and
established a unique-decoding guarantee analogous to the classical bound for
subspace codes. In the constant-dimension setting, we proved that successful
decoding under the classical subspace distance remains sufficient for successful
decoding under the atomic metric, ensuring compatibility with existing theory.
At the same time, we demonstrated through an explicit mixed-dimension example
that the atomic metric can strictly outperform dimension-based decoding by
resolving ambiguities that are invisible to classical distances.

From a broader perspective, our results highlight that redundancy in subspace
coding need not be expressed solely through dimension or ambient space expansion.
Internal combinatorial redundancy—manifested through multiple minimal atomic
representations—can be exploited to improve discrimination and robustness in
settings where degradation acts sparsely or directionally. The atomic framework
developed here preserves the tractability of the subspace coding paradigm while
exposing structural features that are otherwise collapsed by dimension-based
abstractions.

Several directions for future work emerge naturally. These include the design of
explicit code families optimized for the atomic metric, the development of
efficient decoding algorithms that exploit atomic structure, and a deeper
investigation of trade-offs between internal redundancy and code rate. More
broadly, the atomic viewpoint suggests new ways to refine channel models and
metrics in non-coherent communication, bridging combinatorial structure and
network coding practice.

\end{document}